\documentclass[12pt]{amsart}
\usepackage{epsfig,color}

\makeatother

\theoremstyle{definition}

\def\fnum{equation} 
\newtheorem{Thm}[\fnum]{Theorem}
\newtheorem{Cor}[\fnum]{Corollary}

\newtheorem{Lem}[\fnum]{Lemma}
\newtheorem{Con}[\fnum]{Conjecture}

\newtheorem{Exa}[\fnum]{Example}
\newtheorem{Rem}[\fnum]{Remark}
\newtheorem{Pro}[\fnum]{Proposition}

\numberwithin{equation}{section}

\newcommand{\Vol}{{\text{Vol}}}

\newcommand{\dist}{{\text {dist}}}

\newcommand{\distp}{{\text{dist}}_P}

\def\RR{{\bold R}}

\def\SS{{\bold S}}

\newcommand{\e}{{\text {e}}}

\newcommand{\bC}{{\bold{C}}}

\newcommand{\cC}{{\mathcal{C}}}

\newcommand{\cI}{{\mathcal{I}}}

\newcommand{\cH}{{\mathcal{H}}}

\newcommand{\cP}{{\mathcal{P}}}
\newcommand{\cPB}{{\mathcal{P}}B}
\newcommand{\cPT}{{\mathcal{P}}T}
\newcommand{\cPH}{{\mathcal{PH}}}
\newcommand{\cS}{{\mathcal{S}}}

\newcommand{\eqr}[1]{(\ref{#1})}

\title[Space-time regularity of the singular set]{The singular set of mean curvature flow with generic singularities}

\author{Tobias Holck Colding}%
\address{MIT, Dept. of Math.\\
77 Massachusetts Avenue, Cambridge, MA 02139-4307.}
\author{William P. Minicozzi II}%

\thanks{The  authors
were partially supported by NSF Grants DMS  11040934, DMS 1206827,  and NSF FRG grants DMS 
 0854774 and DMS 0853501}

\email{colding@math.mit.edu and minicozz@math.mit.edu}

\begin{document}

\maketitle

\begin{abstract}
 A mean curvature flow starting from a  closed embedded hypersurface in $\RR^{n+1}$ must develop singularities.  We show that if the flow has only generic singularities, then the space-time singular set is contained in   finitely many compact embedded $(n-1)$-dimensional Lipschitz submanifolds plus a set of dimension at most $n-2$.  If the initial hypersurface is mean convex, then all singularities are generic and the results apply.
 
 In $\RR^3$ and $\RR^4$, we show that for almost all times the evolving hypersurface is completely smooth and any connected component of the singular set is entirely contained in a time-slice.  For $2$ or $3$-convex hypersurfaces in all dimensions, the same arguments lead to the same conclusion:  the flow is completely smooth at almost all times and connected components of the singular set are contained in time-slices.
 
A key technical point is a  strong {\emph{parabolic}} Reifenberg property that we show in all dimensions and for all flows with only generic singularities.   We also show that the entire flow clears out very rapidly after a generic singularity.

 These results are essentially optimal.
\end{abstract}

\section{Introduction}
A major theme in PDE's over the last fifty years has been understanding singularities and the set where singularities occur. In the presence of a scale-invariant monotone quantity, 
blowup arguments can often be used to bound the dimension of the 
singular set; see, e.g., \cite{Al}, \cite{F2}.
Unfortunately, these dimension bounds say little about the 
structure of the set.   In this paper we obtain a rather complete description of the singular set for a non-linear evolution equation that originated in material science in the 1920s.  

The  evolution equation is the mean curvature flow (or MCF) of hypersurfaces.  A hypersurface in $\RR^{n+1}$ evolves over time by  MCF if it is locally moving in the direction of steepest descent for the volume element. This equation has been used and studied in material science to model things like cell, grain, and bubble growth.

Under MCF surfaces contract and eventually become extinct.  Along the flow singularities develop.  For instance, a round sphere remains round but shrinks and eventually becomes extinct in a point.  Likewise, a round cylinder remains round and eventually becomes extinct in a line.  For a torus of revolution, the rotational symmetry is preserved  as the torus shrinks and eventually it becomes extinct in a circle.  In these three examples, the singular set consists of a point, a line, and a closed curve, respectively,
 and, in each case, the singularities occur only at a single time.  The natural question is what happens in general?  Are the above examples representative?  Is the singular set contained in a nice submanifold?

The first step towards understanding singularities, and the singular set, in MCF is  blowup analysis.  In the blowup analysis, 
  a sequence of rescalings at a singularity has a subsequence that
converges weakly  to a limiting blowup (or tangent flow).  A priori different subsequences could give  different limits.
A singularity of a MCF is {\emph{cylindrical}} if a blowup  at the singularity is a multiplicity one shrinking round cylinder $\RR^k\times \SS^{n-k}$ for some $k<n$.{\footnote{  
For many of our results (though not all) one can allow the tangent flow to have multiplicity greater than one (cf. \cite{BeWa}), 
however this higher multiplicity does not occur in the most important cases.}}   If at least one tangent flow is  cylindrical, then all are by  \cite{CIM}; in fact, even the axis of the cylinder is unique by
\cite{CM2}. By \cite{CM1},  generic singularities are cylindrical.  Moreover, if the initial hypersurface is mean convex, then all singularities are cylindrical; see, \cite{W1}, \cite{W2}, \cite{W6}, \cite{H2}, \cite{HS1}, \cite{HS2}, \cite{HaK},  \cite{An}.

\vskip2mm
Our main result is that the singular set is rectifiable:

\begin{Thm}	\label{t:main}
Let $M_t \subset \RR^{n+1}$ be a MCF with only cylindrical singularities starting at a  closed smooth embedded hypersurface, then the space-time singular set $\cS$ satisfies:
\begin{itemize}
\item  $\cS$ is contained in finitely many (compact) embedded Lipschitz submanifolds each of dimension at most $(n-1)$ together with a set of dimension at most $(n-2)$.
\end{itemize}
\end{Thm}

This theorem is even stronger than one might think since it uses  the parabolic distance  on space-time $\RR^{n+1}\times \RR$. 
The {\emph{parabolic distance}} between   points $(x,s)$
and $(y,t)$ of $\RR^{n+1}\times \RR$ is  
\begin{equation}       \label{e:dp}
    \distp \, ((x,s),(y,t)) = \max \, \{ |x-y| \, , \, |s-t|^{1/2} \} \, .
\end{equation}
This distance scales differently in time versus space and 
the parabolic distance can be much greater than the Euclidean distance for points at nearby times.  
The parabolic Hausdorff dimension is the Hausdorff dimension with respect to parabolic distance.   In particular, 
   time   has dimension two 
 and  space-time   has dimension $n+3$.

Each submanifold   in Theorem \ref{t:main}
is   the   image of a map from a domain{\footnote{The map can be   taken to be a graph over a subset of a time-slice; this is connected to Theorem \ref{t:main2}.}}
 in $\RR^{n-1}$ to $\RR^{n+1} \times \RR$ that is Lipschitz with respect to   Euclidean distance on $\RR^{n-1}$ and   parabolic  distance  on  $\RR^{n+1} \times \RR$.  The
 proof of Theorem \ref{t:main}  relies crucially on uniqueness of tangent flows.

We prove considerably more than what is stated in Theorem \ref{t:main}; see  Theorem \ref{t:detail} in Section \ref{s:4}.  For instance,  we show regularity of the entire stratification of the space-time singular set.  Moreover, we do so without ever discarding {\bf{any}} subset of measure zero of any dimension  as is always implicit in any definition of rectifiable.{\footnote{See, for instance,
 \cite{D}, \cite{F1}, \cite{P}, \cite{ChCT}, \cite{LjT}, \cite{LW}, \cite{S2}.}}  To illustrate the much stronger version, consider the case of evolution of surfaces in $\RR^3$.  In that case, we show that the space-time singular set is contained in finitely many (compact) embedded Lipschitz curves with cylinder singularities together with a countable set of spherical singularities.   In higher dimensions, we show the direct generalization of this.  

\vskip2mm
In the simple examples of shrinking cylinders and tori of revolution, all of the singularities occurred at a single time.  Part (B) in the next theorem gives criteria to explain this.

\begin{Thm}	\label{t:main2}
If $M_t$ is as in Theorem \ref{t:main},  then   $\cS$ satisfies:
\begin{enumerate}
 \item[(A)] $\cS$ is the countable union of   graphs $(x,t(x))$ of  $2$-H\"older functions on subsets of space.
\item[(B)]  Each subset of $\cS$ with finite parabolic $2$-dimensional Hausdorff measure misses almost every time; each such connected subset is contained in a time-slice. 
\end{enumerate}
\end{Thm}

In  (A),    $x$ lies in  a subset  $\Omega$ of space $\RR^{n+1}$ and the time
 function $t: \Omega \to \RR$   is $2$-H\"older.  Recall that a function is $p$-H\"older if there is a constant $C$ so that 
 \begin{align}
 	|t(x) - t(y)| \leq C \, |x-y|^p \, .
\end{align}
  To say that $t=t(x)$ is  $2$-H\"older  
is equivalent to that the  map $x \to (x, t(x))$ into space-time is Lipschitz with respect to the parabolic metric.  
The $2$-H\"older condition is   very strong  and  is rarely considered since
any $2$-H\"older function on an interval must be constant.  However, there are non-constant $2$-H\"older functions on more general subsets,   including 
disconnected subsets, even of $\RR$,  as in Figure \ref{f:2graph}.
Part (B)  shows   constancy of the time function for connected subsets with finite $2$-dimensional measure.

  \begin{figure}[htbp]		
\centering\includegraphics[totalheight=.4\textheight, width=.9\textwidth]{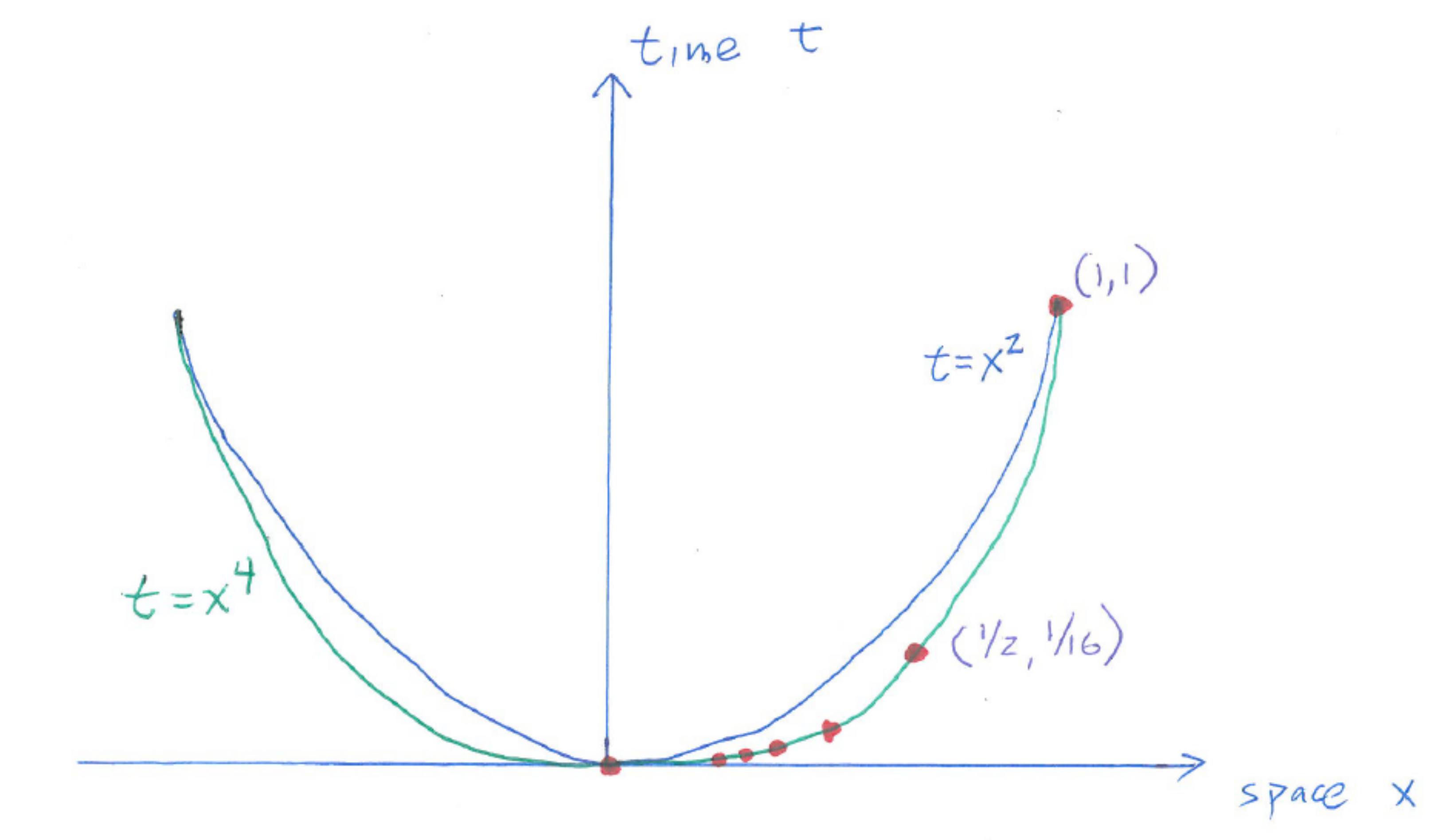}
\caption{An  infinite set that is a $2$-H\"older graph of a non-constant function:  the points $(k^{-1} , k^{-4})$ for $k=1,2, \dots$ and the limit point $0$.}
 \label{f:2graph}
  \end{figure}

 \vskip2mm
Theorems \ref{t:main} and \ref{t:main2} have the following corollaries:

\begin{Cor}	\label{c:main}
Let $M_t \subset \RR^{n+1}$   be a MCF  starting at a  closed smooth embedded mean convex hypersurface, 
 then the conclusions of Theorems
 \ref{t:main} and \ref{t:main2} hold.
\end{Cor}

In dimension three and four we get in addition:

\begin{Cor}  \label{c:main2}
If $M_t$ is as in Theorem \ref{t:main} and $n=2$ or $3$, then the evolving hypersurface is completely smooth (i.e., without any singularities) at almost all times.   
In particular, any connected subset of the space-time singular set is completely contained in a time-slice.  
\end{Cor}

\begin{Cor}  \label{c:main3}
For a generic MCF in $\RR^3$ or $\RR^4$, or a flow starting at a closed  smooth embedded mean convex hypersurface in $\RR^3$ or $\RR^4$,
 the conclusion of Corollary \ref{c:main2} holds.  
\end{Cor}

We get the the same result as in Corollary \ref{c:main3} in all dimensions if we assume that the initial   hypersurface is $2$- or $3$-convex.   A hypersurface is said to be $k$-convex if the sum of any $k$ principal curvatures is nonnegative.  

\subsection{Some ingredients in the proof}

A key technical point in this paper is to prove a strong parabolic Reifenberg property for MCF with generic singularities.  In fact, we will show that the space-time singular set is parabolic Reifenberg vanishing.  In analysis\footnote{See, for instance, \cite{T}.} a subset of Euclidean space is said to be Reifenberg (or Reifenberg flat) if on all sufficiently small scales it is, after rescaling to unit size, close  to a $k$-dimensional plane.  The dimension of the plane is always the same but the plane itself may change from scale to scale and from point to point.  Many fractal curves, like the Koch snowflake, are Reifenberg with $k=1$ but have Hausdorff dimension strictly larger than one; see Figure \ref{f:koch}.  A set is said to be Reifenberg vanishing if the closeness to a $k$-plane goes to zero as the scale goes to zero.  It is said to have the strong Reifenberg property if the $k$-dimensional plane depends only on the point but not on the scale.  Finally, one sometimes distinguishes between half Reifenberg and full Reifenberg, where half Reifenberg refers to that the set is close to a $k$-dimensional plane, whereas full Reifenberg refers to that in addition one also has the symmetric property: The plane  is also close to the set on the given scale.  

  \begin{figure}[htbp]		
\centering\includegraphics[totalheight=.4\textheight, width=.5\textwidth]{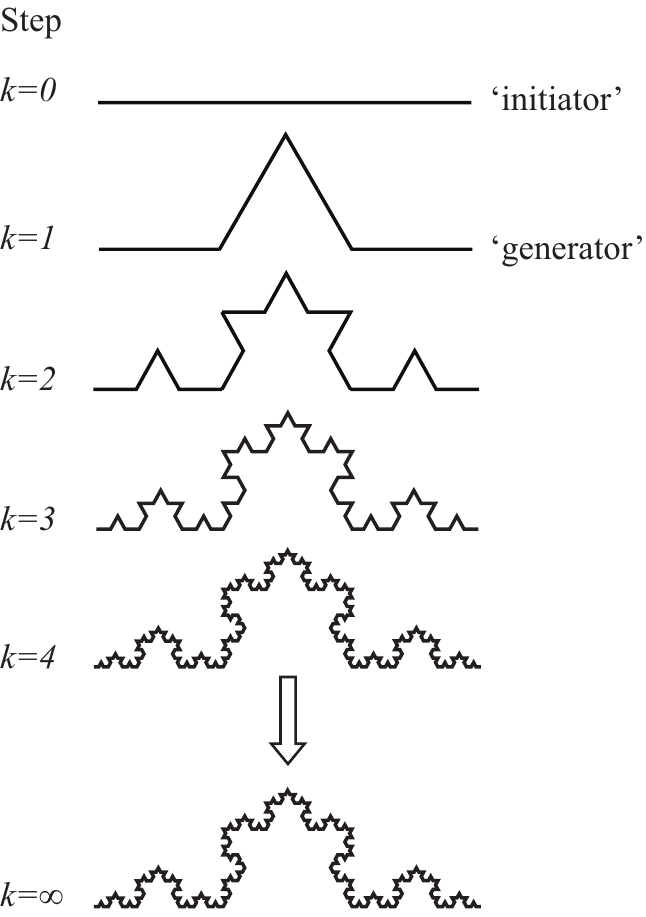}
\caption{The Koch curve is  close to a line on all scales, yet the line that it is close to changes from scale to scale.  It is not rectifiable but admits a H\"older parametrization.}
 \label{f:koch}
  \end{figure}

Using \cite{CM2},  we show in this paper that the singular set in space-time is strong (half) Reifenberg vanishing with respect to  parabolic   distance.     
\subsection{Comparison with prior work}

The results of this paper should be contrasted with a result of Altschuler-Angenent-Giga, \cite{AAG} (cf. \cite{SrSs}).  The paper
\cite{AAG}  showed that in $\RR^3$ the evolution of any rotationally symmetric surface obtained by rotating the graph of a function $r = u(x)$, $a < x < b$ around the $x$-axis is smooth except at finitely many points in space-time where either a cylindrical or spherical singularity forms.  For more general rotationally symmetric surfaces (even mean convex), the singularities can consist of nontrivial curves.  For instance, consider a torus of revolution bounding a region $\Omega$. If the torus is thin enough, it will be mean convex.  Since the symmetry is preserved and because the surface always remains in $\Omega$, it can only collapse to a circle. Thus at the time of collapse, the singular set is a simple closed curve.  In \cite{W1}-\cite{W4} (see, for example, section $5$ of \cite{W3}), White showed that a mean convex surface evolving by MCF in $\RR^3$ must be smooth at almost all times, and at no time can the singular set be more than $1$-dimensional.   
In all dimensions, White, \cite{W1}--\cite{W6}, 
showed that the space-time singular set of a mean convex MCF has parabolic Hausdorff dimension at most $(n-1)$; see also theorem $1.15$ in \cite{HaK}.
White's dimension bounds are proven by classifying the blowups and then appealing to his parabolic version of Federer's dimension reducing argument in [W5].  The dimension reducing 
  gives that the singular set of any MCF with only cylindrical singularities has dimension at most $(n-1)$.

\vskip2mm
We conjecture:

\begin{Con}
Let $M_t \subset \RR^{n+1}$ be a MCF with only cylindrical singularities starting at a closed smooth embedded hypersurface.  Then the 
space-time singular set has only finitely many components.  
\end{Con}

If this conjecture is true, then it would follow from this paper that in $\RR^3$ and $\RR^4$ MCF
 with only generic singularities is smooth except at finitely many times; cf.  \cite{Ba} and  section $5$ in \cite{W3}.
 
 \vskip2mm
 Each time-slice of a MCF will be a subset of $\RR^{n+1}$, but the space-time track of the flow is a subset of 
 $\RR^{n+1}\times \RR$, where the first $n+1$ coordinates are in space and the last is the time variable.   
With the parabolic distance \eqr{e:dp}, the unit  {\emph{parabolic  ball}} $\cPB_1 (0,0)$ at $x=0$ and
$t=0$ is the product
 of $ B_1 (0)  \subset \RR^{n+1} $ and the unit interval
$(-1,1) \subset \RR$.  Similarly, 
  a parabolic ball of radius $r$ is given by a translated copy of $ B_r (0) \times
(-r^2,r^2)$.  This scaling   implies that the volume of a parabolic
ball of radius $r$ is a constant times $r^{n+3}$.   
Finally, 
 $\cPT_{r} (T)$
will be the {\emph{parabolic tubular neighborhood}} of radius $r$ about a set $T \subset \RR^{n+1}\times \RR$  
\begin{align}
	\cPT_{r} (T) = \{ y \in \RR^{n+1}\times \RR \, | \, \distp (y , T) < r \} \, .
\end{align}

 \vskip2mm
The results proven here are used in \cite{CM6} to prove regularity for the arrival time in the level set method.

\section{Parabolic Reifenberg}

A subset of space $\RR^{n+1}$ is  Reifenberg{\footnote{See \cite{CK}, \cite{DKT}, \cite{HoW}, section $2.3$ of \cite{LY}, \cite{PTT}, \cite{R}, \cite{S1}, and \cite{T}.}}
  if it is close  to some $k$-dimensional plane on all sufficiently small scales.  The $k$-plane can vary from point to point and from scale to scale and  ``close'' means that 
the Hausdorff distance between the set and the $k$-plane is small relative to the scale.  The Koch curve (with small angle) has this property; see Figure \ref{f:koch}.
  If the $k$-plane does not depend on the scale, then the set is {\emph{strong Reifenberg}}.
 We will need a parabolic version   for subsets of space-time, where the parabolic   distance  $\distp$ is used in place of Euclidean distance; see Figure \ref{f:Reif}.

\subsection{Strong parabolic Reifenberg}
We say that a subset $S\subset \RR^{n+1}\times \RR$ has the {\emph{strong parabolic $k$-dimensional Reifenberg property}} if for some small $\delta \in (0,\frac{1}{8})$ and some $r_0>0$:  For all $0<r\leq r_0$, $(x,t)\in S$, there is a $k$-plane $V_{x,t} \subset \RR^{n+1} \times \{ t \}$    so that
\begin{align}	\label{e:halfReif}
     \cP B_r (x,t) \cap S  {\text{ is contained in the $\delta \, r$ parabolic tubular neighborhood of }}    V_{x,t}   \, .
\end{align}

This could be called a \underline{half}\footnote{See, for instance, the remarks on page 258 of \cite{S1}.} Reifenberg because \eqr{e:halfReif} requires only that $S$ is contained a tube about the plane $V_{x,t}$; the \underline{full} Reifenberg is symmetric and requires also that the plane is contained in a tube around $S$.
We emphasize that the $k$-plane  $V_{x,t}$ is contained in $\RR^{n+1} \times \{t\}$   
and is allowed to depend only on the point $(x,t)$ in space-time, but not  on the scale $r$.     The required closeness in \eqr{e:halfReif} is proportional to the scale.

  \begin{figure}[htbp]		
\centering\includegraphics[totalheight=.55\textheight, width=1.05\textwidth]{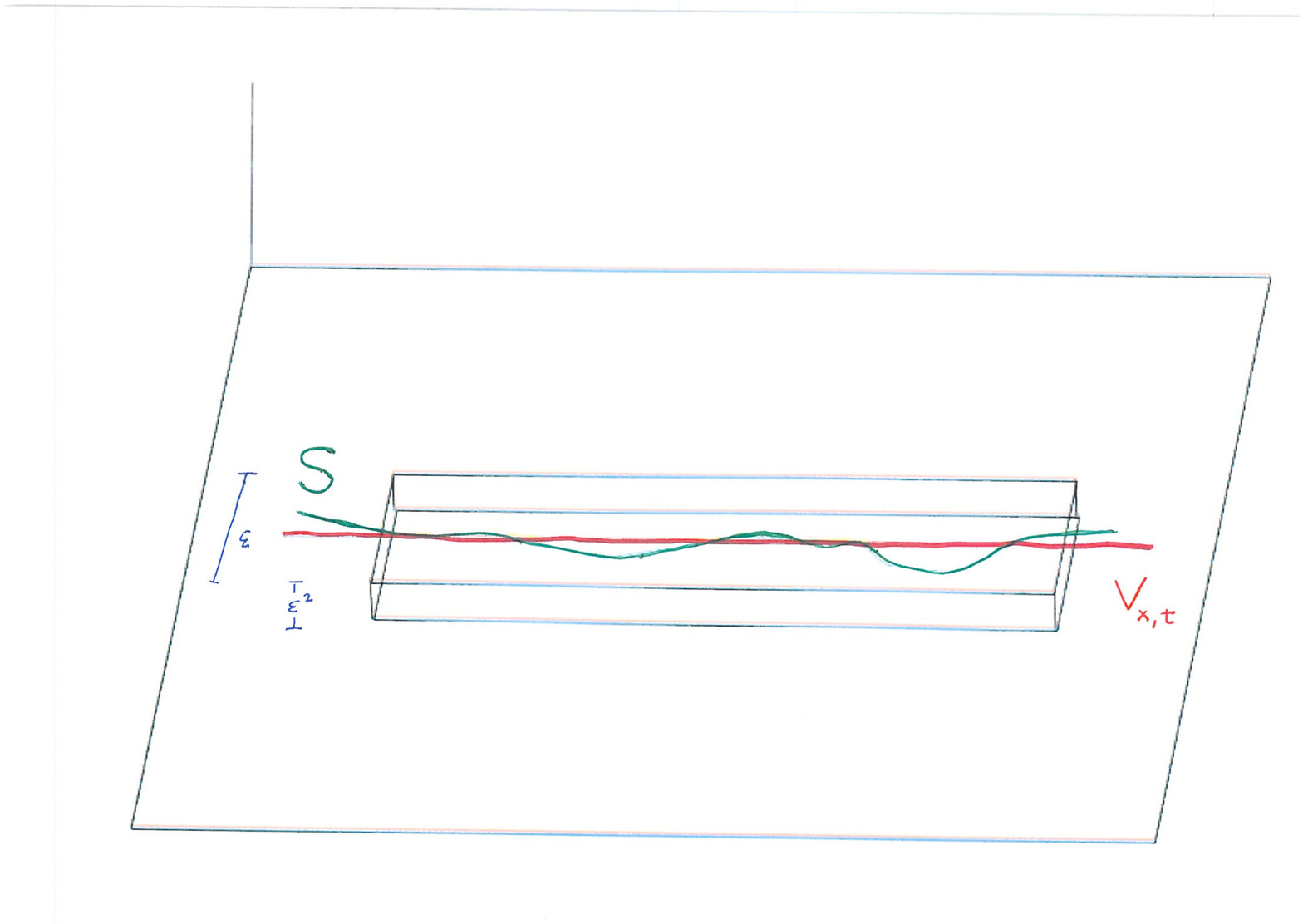}
\caption{The parabolic Reifenberg property illustrated on one scale: The green set $S$ lies in a parabolic $\epsilon$-tubular neighborhood of the red $k$-dimensional plane $V_{x,t}$.}
 \label{f:Reif}
  \end{figure}

It is instructive to keep in mind that a line of the form $t=a\, x$ in space-time $\RR \times \RR$  is parabolic one-dimensional Reifenberg  only if $a=0$.  More generally, any $C^1$ connected curve   satisfying the parabolic one-dimensional Reifenberg condition must be   in a time-slice.

We say that a subset $S\subset \RR^{n+1}\times \RR$ is {\emph{Reifenberg vanishing}} if $\delta=\delta (r)\to 0$ as $r\to 0$.  

 The next example is strong $1$-Reifenberg on scales $< 1$ and strong $2$-Reifenberg on larger scales.  It is $1$-Reifenberg on all scales, but not strong.

\begin{Exa}	\label{ex:1}
(The following example is in space; there are similar examples in space-time.)

The set consisting of the four points $(0,0,0)$, $(1,0,0)$, $(0,\epsilon,0)$, $(1,0,\epsilon)$ in $\RR^3$ (see   Figure \ref{f:example})
satisfies the strong $2$-dimensional Reifenberg property on the scale $2$ (and, in fact, on all scales) with 
$\delta = \epsilon$.    
The approximating two-planes can be taken to be $V_{(0,0,0)} = V_{(0,\epsilon,0)} = \{ z = 0\}$ and $V_{(1,0,0)} = V_{(1,0,\epsilon)} = \{ y= 0 \}$.   
\begin{itemize}
\item
The four points {\underline{do not}} satisfy the strong $2$-dimensional {\emph{full}} Reifenberg property: the points are locally contained in  tubular neighborhoods of the planes, but the planes are not locally contained in neighborhoods of the points  because of the gaps in the set.
\item
The four points   {\underline{do not}}  have the strong $1$-dimensional Reifenberg property at scale $2$  for $\delta = \epsilon$.  To see this, observe that the approximating line at $(0,0,0)$  must be close to the $y$-axis at scale $\epsilon$ and  close  to  the $x$-axis at scale $2$.  Thus,   no single line  works at both scales.
 \end{itemize}
 
 \end{Exa}

The next lemma will give a condition that forces Reifenberg sets to be bi-Lipschitz graphs.  Some additional condition is necessary since
 Example \ref{ex:1} satisfies the strong $2$-dimensional Reifenberg property on the scale $2$, but there is no single $2$-plane where the projection is bi-Lipschitz with constant close to one.  The extra condition is \eqr{e:23} in the next lemma.  This will give that the planes for nearby points are close, giving   additional regularity of the set.   In the lemma, $d_{\cP H}$ denotes  parabolic Hausdorff distance.

\begin{Lem}	\label{l:2p2}
There exists $\delta>0$ such that if 
$y_0\in S\subset \RR^{n+1}\times \RR$, $S$ has the strong $(\delta , r_0)$-Reifenberg property and 
\begin{align}	\label{e:23}
d_{\cP H}(\cP B_{r_0}(y_0)\cap S,\cP B_{r_0}(y_0)\cap V_{y_0})<\delta \, r_0 \, . 
\end{align}
then the projection $\pi_{y_0} :\cP B_{\frac{r_0}{2}}(y_0)\cap S\to V_{y_0}$ is a bi-Lipschitz map to its image. 
This implies the set  is a graph $x \to (x,t(x))$ over a subset of space and
  $x \to t(x)$ is  $2$-H\"older  on space.
\end{Lem}

\begin{proof} 
To keep the notation simple,  translate so that
$y_0 = (0,0)$ and rescale so that $r_0 = 1$.   We will omit the $(0,0)$ below and  write $V$ for $V_{0,0}$,  $\pi$ for $\pi_{0,0}$, and $\cP B_r$ for $\cP B_r (0,0)$.  Given any $y \in \RR^{n+1}$, we can decompose it into
\begin{align}
	y =   \pi (y) + y^{\perp}  
\end{align}
where      $y^{\perp} \in \RR^{n+1}$ is the   part orthogonal to $V$.

The projection $\pi$ is  Lipschitz, so we must show that $\pi$ is one-to-one and the inverse is  Lipschitz.   This will follow by showing that if
 $(y,t)$ and 
$(z,s) \in \cP B_{\frac{1}{2}}\cap S$, then 
 \begin{align}	\label{e:2p2a}
 	\left| t- s \right|^{ \frac{1}{2} } &\leq C\, \left|  \pi (y)   - \pi (z) \right| \, , \\
	\label{e:2p2b}
 	\left|  y^{\perp} - z^{\perp} \right| &\leq C\, \left|  \pi (y)   - \pi (z)  \right| \, ,
\end{align}
 where the constant $C$ depends only on $\delta$ and $n$.

We will show that \eqr{e:2p2a} and \eqr{e:2p2b} follow from the strong Reifenberg property plus
\begin{enumerate}
	\item[($\dagger$)]	If $(y,t) \in \cP B_{ \frac{1}{2} }\cap S$, then $B_{\frac{1}{2}} (y) \cap V_{y,t} \subset  T_{4\delta}(V +y)$,
\end{enumerate}
where $T_s(K)$ is the spatial tubular neighborhood of radius $s$ about a subset $K \subset \RR^{n+1}$ and we have identified $V_{y,t} \subset \RR^{n+1} \times \{ t\}$ with the parallel plane in $\RR^{n+1}$.

{\it{Proof of \eqr{e:2p2a} and \eqr{e:2p2b} assuming $(\dagger)$}}:   If we let $r=\distp ( (y,t) , (z,s))$,  then  
\begin{align}	\label{e:zhere}
	(z,s) \in \cP T_{\delta \,r}(V_{y,t})\, .
\end{align}
  Since $V_{y,t}$ is contained in a time-slice, 
this implies that
\begin{align}	\label{e:whatgi}
	\left| t-s \right|^{ \frac{1}{2} }  \leq \delta \, r = \delta \, \max \,  \left\{ \left| t-s \right|^{ \frac{1}{2} }  , |y-z|
		\right\} \, .
\end{align}
Since $\delta < 1$, we see that $r = |y-z|$.
Scaling the conclusion in ($\dagger$) gives   
\begin{align}
	B_{r} (y) \cap V_{y,t} \subset  T_{8\delta\, r}(V +y) \, . 
\end{align}
Combining this with \eqr{e:zhere}   gives
\begin{align}
	\pi (z)  + z^{\perp} = z  \in \overline{B_r(y)} \cap T_{\delta \,r}(V_{y,t}) \subset  T_{9 \delta \,r}(V +y)   = T_{9 \delta \,r}(V +y^{\perp})  \, ,
\end{align}
where we used that   $V$ is invariant under translation by the vector $\pi (y) \in V$.  Thus, we get 
\begin{align}
	  z^{\perp} \in    T_{9 \delta \,r}(V +y^{\perp})     \, ,
\end{align}
which implies that $\left| z^{\perp} - y^{\perp} \right| \leq 9 \delta \,r$.  As long as $\delta < \frac{1}{9}$, this implies \eqr{e:2p2b}.  
 Equation \eqr{e:2p2a} follows immediately from \eqr{e:2p2b} and \eqr{e:whatgi}.

{\it{Proof of   $(\dagger )$}}:
 Use \eqr{e:23} and the strong Reifenberg property at $y$ to get
\begin{align}
	  \cP B_{ \frac{1}{2} }\cap V \subset    \cP B_{ 1 }(y,t) \cap  \cP T_{\delta}(S) \subset    \cP T_{2\delta}(V_{y,t}) \, .
\end{align}
This implies that $V_{y,t}$ intersects $B_{2\delta}$, so we get $T_{2\delta}(V_{y,t}) \subset T_{4\delta} (V_{y,t} - y)$ and, thus,
\begin{align}
	  B_{ \frac{1}{2} }\cap V  \subset    T_{4\delta}(V_{{y,t}} - y) \, .
\end{align}
Finally, since   $\dim V = \dim V_{{y,t}}$, 
we can apply Lemma \ref{l:linalg} below to get  ($\dagger$).

\end{proof}

 \begin{Lem}	\label{l:linalg}
 Suppose that $V,W$ are both $k$-planes through $0$ in $\RR^{n+1}$.  If $B_1 \cap V \subset T_{\delta}(W)$ for some $\delta \in (0,1)$, then 
 $B_1 \cap W \subset T_{\delta}(V)$.
 \end{Lem}
 
 \begin{proof}
 Let $\Lambda: V \to W$ be orthogonal projection onto $W$.  Thus, if $v \in V$, then
 \begin{align}	\label{e:myvw1}
 	 \left| v- \Lambda (v) \right| = \dist (v , W)  \leq \delta \, |v| \, , 
 \end{align}
 where the   inequality used that $B_1 \cap V \subset T_{\delta}(W)$.  It follows that
 \begin{align}	\label{e:myvw2}
 	 \left| \Lambda (v) \right|^2 = |v|^2 -   \left| v-\Lambda (v) \right|^2 \geq (1- \delta^2) \, |v|^2 \, .
 \end{align}
 We conclude that the linear map $\Lambda$ is injective.  Since $V$ and $W$ have the same dimension, $\Lambda$ is also   onto.  Thus, given any $w \in B_1 \cap W$, there exists $v \in V$ with
 $\Lambda (v) = w$.  Finally, by  \eqr{e:myvw1}, the sine of the angle between $v$ and $w$ is at most $\delta$ and, thus, the distance from $w$ to the line through $v$ is at most $\delta$.

 \end{proof}

\vskip2mm
In Example \ref{ex:1}, the two pairs of points  could be separated   by working on scales less than one so that only one pair was visible at a time.   This is not possible in the next example, where we have arbitrarily  close points with approximating planes that are very different.  This example will be strong  $1$-Reifenberg, but not on a uniform scale.

\begin{Exa}
(As above, the following example is in space; there are similar examples in space-time.)
The set consisting of the union of the three sequences 
$(\epsilon^{n},0,0)$, $(\epsilon^{2n},\epsilon^{2n+1},0)$, $(\epsilon^{2n+1},0,\epsilon^{2n+2})$ satisfies the strong Reifenberg property for $\epsilon>0$ sufficiently small.  This set contains
two sequences which converge to the origin inside two planes which are perpendicular to each other.
\end{Exa}

 \begin{figure}[htbp]		
\centering\includegraphics[totalheight=.4\textheight, width=.9\textwidth]{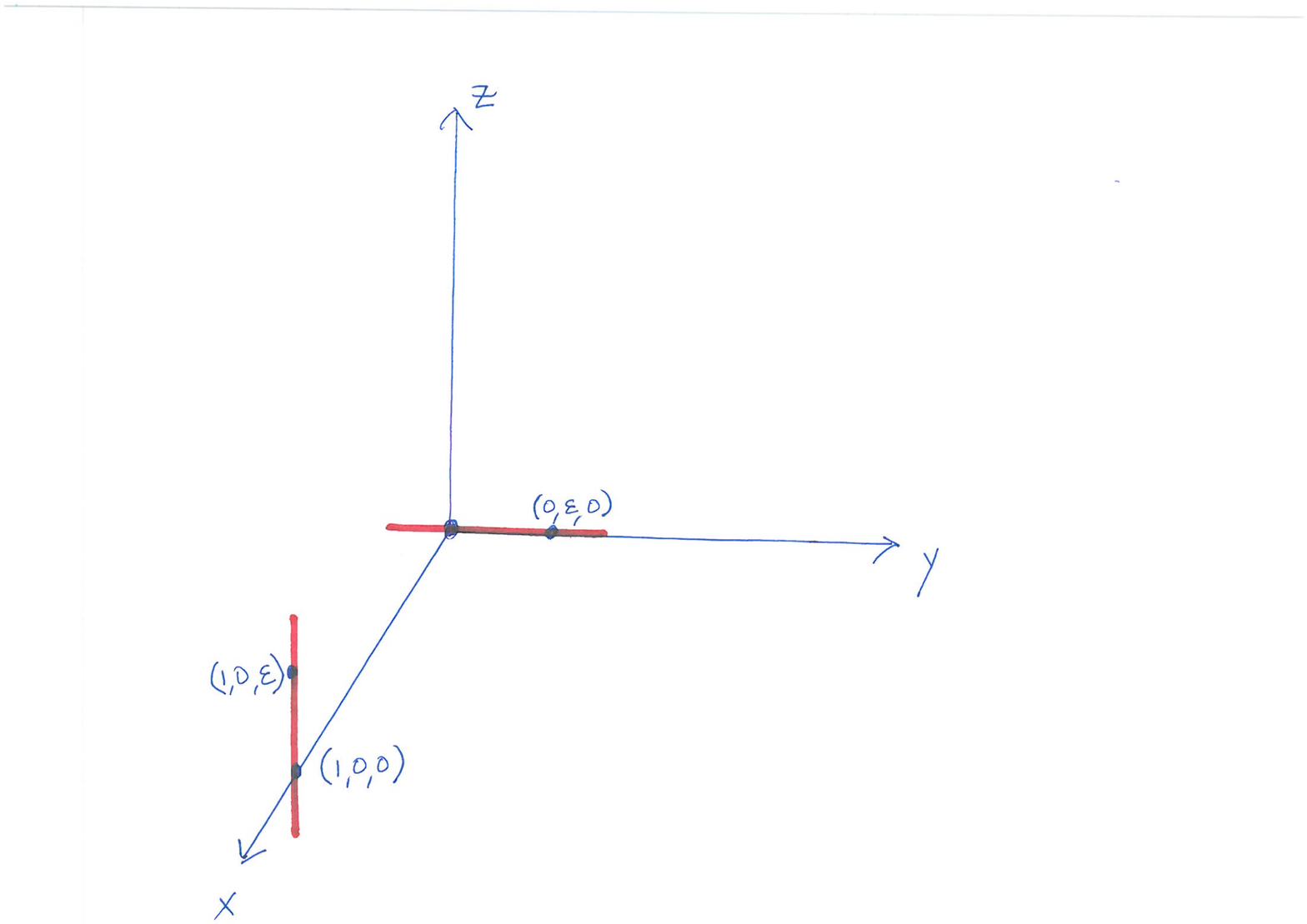}
\caption{The set $\{ (0,0,0), (1,0,0) , (0,\epsilon,0),(1,0,\epsilon)\}$  satisfies the strong $2$-dimensional Reifenberg property but there is no  $2$-plane where the projection is bi-Lipschitz with constant close to one.}
 \label{f:example}
  \end{figure}

In the applications that we have in mind, we will not able to appeal to Lemma \ref{l:2p2}.  However, 
 the distribution $\{V_{y,t}\}$ of $k$-planes in our applications will have an additional regularity property which, as we will see,  implies regularity of the set $S$.  Roughly speaking, this property is equi-continuity of the   distribution $\{V_{y,t}\}$:

Suppose we have a distribution of $k$-dimensional  planes $\{V_{y,t}\}$ with $V_{y,t} \subset \RR^{n+1} \times \{ t \}$
 labeled by a subset $S\subset \RR^{n+1}\times \RR$ with $(y,t)\in V_{y,t}\cap S$.  For some $\delta>0$ sufficiently small, let $f:(0,1)\to (0,\delta)$ be a monotone non-decreasing function with $\lim_{r\to 0} f(r)=0$.  We will say the distribution $\{V_{y,t}\}$ is {\emph{$f$-regular}} if for all $r>0$ and all $(y_1,t_1)$, $(y_2,t_2)\in S$ 
\begin{align}
d_{\cP H}(\cP B_r(y_1,t_1)\cap V_{y_1,t_1},\cP B_r(y_1,t_1)\cap V_{y_2,t_2})<f(r)\,r\, ,
\end{align}
for $r=\distp((y_1,t_1),(y_2,t_2))$.

Later in our applications to MCF with generic singularities,
 we will show that not only does the space-time singular set satisfy the strong parabolic (half) vanishing Reifenberg property but the distribution of $k$-planes will be $f$-regular.{\footnote{In fact, one can take $f(r)\approx (\log |\log r|)^{-\alpha}$ ($\alpha>0$).}}

\vskip2mm
The next theorem is the main result of this section.    The strong Reifenberg property  gives a bi-Lipschitz approximation when the approximating planes are close at nearby points.    In Lemma  \ref{l:2p2},     \eqr{e:23} implies this closeness.  If the distribution is $f$-regular, then the closeness is automatic.

\begin{Thm}  \label{t:Lipgraph}
If $S\subset \RR^{n+1}\times \RR$ satisfies the  
strong parabolic Reifenberg property and the distribution of  $k$-planes is $f$-regular, then for $(y,t)\in S$ fixed the projection $\pi:S\to V_{y,t}$ is a bi-Lipschitz map from a neighborhood of $(y,t)$ in $S$ to its image (equivalently, near $(y,t)$, $S$    is a Lipschitz graph over part of $V_{y,t}$).
\end{Thm}

\begin{proof}
The proof is a slight variation of the proof of Lemma \ref{l:2p2}, where the $f$-regularity gives the condition ($\dagger$) there.  
\end{proof}

\section{Cylindrical tangent flows}

In this section, we first recall the Gaussian surface area, the monotonicity formula for MCF and its basic properties.
After that, we
  record a  consequence of the uniqueness theorem of \cite{CM2} for cylindrical singularities that will be the key to establishing the strong parabolic Reifenberg property for the singular set.  Recall that we will only prove the half Reifenberg, meaning that the set lies in a small tubular neighborhood of a plane but not vice versa.

\subsection{Gaussian surface area}

The $F$-functional, or Gaussian surface area,  of a hypersurface $\Sigma \subset \RR^{n+1}$ is
\begin{align}
	F_{x,\tau} (\Sigma) = \left( 4 \pi \tau \right)^{ - \frac{n}{2} } \, \int_{\Sigma} \e^{ - \frac{ |y-x|^2}{4\tau} } \, dy \, ,
\end{align}
where  the Gaussian is centered  at $x \in \RR^{n+1}$   and $\sqrt{ \tau } >0$ is the scale.  
The entropy $\lambda$ is the supremum over all Gaussians  (i.e., over all centers $x_0$ and scales $\sqrt{t_0}$)
\begin{equation}
  \lambda (\Sigma)=\sup_{x_0 , t_0}   \, \,  F_{x_0,t_0} (\Sigma)  \, .
\end{equation}

 Huisken's monotonicity formula{\footnote{Ilmanen and White 
extended the monotonicity to the case where $M_t$ is a Brakke flow.}}, \cite{H1},  for a MCF $M_t \subset \RR^{n+1}$ states that 
$F_{x,\tau}(M_{t-\tau})$ is increasing in $\tau$ for every fixed $x$ and $t$.  That is,  if $0 < \tau_1 < \tau_2$, then    
\begin{align}	\label{e:huisken}
	F_{x,\tau_1 } (M_{t  -\tau_1})  \leq F_{x,\tau_2 } (M_{t-\tau_2})  \, 
\end{align}
 The monotonicity gives  an upper semi-continuous limiting Gaussian density $\Theta_{x,t}$    
 \begin{align}	\label{e:density}
 	\Theta_{x,t} = \lim_{\tau \to 0_+} F_{x,\tau} (M_{t-\tau}) \, .
 \end{align}
 The limiting density is at least one at each point in the support of the flow and greater than one at each singularity.

By definition, a tangent flow is the limit of a sequence of parabolic dilations 
at a singularity, where the convergence is on compact subsets. For instance, a tangent flow
to $M_t$ at the origin in space-time is the limit of a sequence of rescaled flows $\frac{1}{\delta_i}\,M_{\delta_i^2 \,t}$ where $\delta_i\to 0$. By   Huisken's monotonicity formula and an argument of Ilmanen and White, \cite{I1}, \cite{W4}, tangent flows are shrinkers, i.e., self-similar solutions of MCF that evolve by rescaling. 
A priori, different sequences $\delta_i$ could give different tangent flows; the question of   uniqueness of the blowup is whether the limit is independent of the sequence. 
In [CM2] (see also \cite{CM5}) it was proven that tangent flows at cylindrical singularities are unique. That is, any other tangent flow is also a cylinder with the same $\RR^k$ factor that points in the same direction.

 \subsection{Cylindrical singularities}
 
 Throughout this subsection, $M_t$ will be a MCF in $\RR^{n+1}$ with entropy at most $\lambda_0$.  All constants will be allowed to depend on $n$ and $\lambda_0$.

 We will let
  $\cC_k$ denote the   cylinder  of radius $\sqrt{2(n-k)}$
  \begin{align}
  	\RR^k\times \SS^{n-k}_{\sqrt{2(n-k)}} 
\end{align}
 and its rotations and translations in space.{\footnote{This is a different convention than in \cite{CM2} where we used $k$ for the dimension of the spherical factor.}}

 \vskip2mm
 It will be useful to have compact notation for spatial rescalings of a set about a fixed point.  Namely, given $z \in \RR^{n+1}$, $r> 0$ and   $\Lambda \subset \RR^{n+1}$, let $\Psi_{z,r}(\Lambda)$ be the rescaling of $\Lambda$ about $z$ by the factor $r$, i.e.,
\begin{align}
	\Psi_{z,r}(\Lambda) = \{ r (x-z) + z \, | \, x \in \Lambda \} \, .
\end{align}
The map $\Psi_{z,r}$ is a rescaling in space only.  

\vskip2mm
We need a notion of what it means for the flow to look like a cylinder near a point.  Namely, 
we will say that $M_t$ is
{\emph{$(j,\eta)$-cylindrical at $(y,t)$ on the time-scale $\tau$}}
 if for every positive $s \leq \tau$
\begin{itemize}
	\item $\Psi_{y,\frac{1}{\sqrt{s}}} \, \left( B_{\eta^{-1}\sqrt{s}}(y) \cap M_{t-s} \right)$ is a graph over   a fixed cylinder in $\cC_j$ of a function with $C^1$ norm at most $\eta$.
\end{itemize}

It is important that the dilates  $\Psi_{y,\frac{1}{\sqrt{s}}} \, \left( B_{\eta^{-1}\sqrt{s}}(y) \cap M_{t-s} \right)$ are graphs over the same cylinder as the time scale $s$ varies. 
The constant $j$ here is the dimension of the Euclidean factor of the cylinder.  The constant $\eta$ measures how close the flow is to the   cylinder in a scale-invariant way and, as $\eta$ gets smaller, the flow is closer to a cylinder on a  larger set.

\vskip2mm
We will need the following result from \cite{CM2} which gives that the flow becomes cylindrical near every cylindrical singularity.  Roughly speaking, this says that if the flow is a graph over a cylinder just before a cylindrical singularity, then it remains a graph over the same cylinder as one approaches the singularity.   Moreover, after rescaling, it is a graph over a larger set and is even closer to the cylinder.

\begin{Thm}	\label{t:techmon}
Given $\eta > 0$, there exists $\epsilon > 0$ so that if $(x_0 , t_0)$ is a cylindrical singularity  in $\cC_j$ and
\begin{itemize}
\item[($\cC$)] 	  $\Psi_{x_0,\frac{1}{\sqrt{2\tau}}} \, \left( B_{\epsilon^{-1}\sqrt{2\tau}}(x_0) \cap M_{t_0-2\tau} \right)$ is a graph over  
 {\underline{some}} cylinder
in $\cC_j$ of a function with $C^1$ norm at most $\epsilon$, 
\end{itemize}
then $M_t$ is $(j,\eta)$-cylindrical at $(x_0,t_0)$ on the time-scale $\tau$.  Furthermore, for each $\bar{\eta} > 0$, there exists $\bar{\tau} \in (0,\tau)$ so that
$M_t$ is $(j,\bar{\eta})$-cylindrical at $(x_0,t_0)$ on the time-scale $\bar{\tau}$.
\end{Thm}

\begin{proof}
The first claim follows from  theorem $0.2$ in \cite{CM2}.  See footnotes $5$ and $6$ in \cite{CM2} for the fact that it becomes even more cylindrical 
(i.e., $(j,\bar{\eta})$-cylindrical for $\bar{\eta}$ smaller than $\eta$) at smaller scales, giving the second claim.
\end{proof}

 We will also need a version of $(j,\eta)$-cylindrical where the same time-scale $\tau > 0$ works uniformly on a subset $S \subset \cS$.  Namely,  
$M_t$ is  {\emph{uniformly $(j,\eta)$-cylindrical on $S$  on the time-scale $\tau$}}  if 
$M_t$ is
 $(j,\eta)$-cylindrical at each $y \in S$ on the time-scale $\tau$.
Thus, the axis of the cylinder may vary with $y$, but the time-scale $\tau$ cannot.
 
 In the next corollary, we will fix  $j$ and   let $S$ be the subset of $\cS$ with singularity in $\cC_j$.

\begin{Cor}	\label{c:techmon}
Given $\eta > 0$, $j$,  and  $y \in S$, there exists $r_y > 0$ and $\tau_y > 0$ so that 
\begin{itemize}
\item
 $M_t$ is uniformly $(j,\eta)$-cylindrical on $\cPB_{r_y} (y) \cap S$ on the time-scale $\tau_y$.  
 \item  For   $\bar{\eta} > 0$, there exists $\bar{\tau} >0$ so that
$M_t$ is uniformly $(j,\bar{\eta})$-cylindrical on $\cPB_{r_y} (y) \cap S$ on   time-scale $\bar{\tau}$.
 \end{itemize}
\end{Cor}

\begin{proof}
Observe that if ($\cC$) holds for $\epsilon/2$ in place of $\epsilon$, then it holds for $\epsilon$ at nearby points.   
The corollary follows from this observation and Theorem \ref{t:techmon}.
\end{proof}

 \section{The singular set of a MCF}		\label{s:4}

Let $M_t$ be a MCF flow   in $\RR^{n+1}$  with only cylindrical singularities   that starts at a closed smooth embedded hypersurface.  This means that
 a tangent flow  at any singularity is a multiplicity one shrinking round cylinder $\RR^k\times \SS^{n-k}$ for some $k<n$.  If at least one blowup is  cylindrical, then all are by  \cite{CIM}
 and the axis of the cylinder is unique by
\cite{CM2}.

The singular set $\cS$ is a compact subset of space-time $\RR^{n+1}\times \RR$ and can be stratified into subsets 
\begin{align}
	\cS_0 \subset  \cS_1 \subset  \dots \subset \cS_{n-1} = \cS \, .
\end{align}
The set   $\cS_k$ consists of all singular points where the tangent flow splits off a Euclidean factor of dimension at most $k$.  Thus, 
\begin{itemize}
\item $y \in \cS_0$ if the tangent flow at $y$ is an $n$-sphere.
\item $y \in \cS_{k} \setminus \cS_{k-1}$ if the tangent flow at $y$ is in $\cC_{k}$.
\end{itemize}
The Gaussian density at a singularity where the tangent flow is in $\cC_k$ is equal to $\Theta_k \equiv F_{0,1}(\cC_{k})$.  The $\Theta_k$'s    are increasing in $k$  with
\begin{align}	\label{e:increasing}
      1 < \Theta_0 < \Theta_1 < \dots < \Theta_{n-1} \, .
\end{align}
Therefore, each strata $\cS_k \setminus \cS_{k-1}$ is  characterized by the value of the Gaussian density   at the singular point.  Namely, $(x,t) \in \cS_k \setminus \cS_{k-1}$ if and only if
$\Theta_{x,t}$ is equal to $\Theta_k$.
 Moreover, by the upper semi-continuity of the Gaussian density and  \eqr{e:increasing},    $\cS \setminus \cS_{k-1}$ is compact for each $k$.
 For example, in $\RR^3$, a limit of  cylindrical singular points must   be cylindrical, while a limit of spherical points must be singular but, a priori, could be either spherical or cylindrical.

We have seen that the top strata $\cS_{n-1} \setminus \cS_{n-2}$  is compact.  The same holds for the lower strata  as long as we stay away from the higher strata:

\begin{Lem}	\label{l:unicone}
The set $\cS_{n-1} \setminus \cS_{n-2}$ is compact. 
Moreover, if $\epsilon > 0$ and $k \in \{ 0 , \dots , n-2 \}$, then  $\left( \cS_k \setminus \cS_{k-1} \right) \setminus \cPT_{\epsilon} \left( \cS \setminus \cS_k \right)$ is compact (with the convention that $\cS_{-1} = \emptyset$). 
\end{Lem}

 \begin{proof}
 This follows immediately since $\cS \setminus \cS_{k-1}$ is compact for each $k$.
 \end{proof}

We observe next that the lowest strata $\cS_0$ consists of isolated points.

\begin{Lem}	\label{l:isolatedbottom}
The set $\cS_0$ is isolated: for each $y \in \cS_0$, there is a space-time ball $\cPB_{r_y}(y)$ so that $\cPB_{r_y} (y) \cap \cS = \{ y \}$.
\end{Lem}

\begin{proof}
By assumption, a multiplicity-one sphere is collapsing off at  $y$.  Thus, by Brakke's regularity theorem, this sphere is itself a connected component of the flow just before the singular time.  In particular, the flow {\emph{take away this sphere}} has $F$-functional much less than one just before the singularity on the scale of the singularity.  Monotonicity gives a space-time neighborhood  where the densities (after we take away the sphere) are less than one and, thus, there are no other singularities.
\end{proof}

One immediate consequence of Lemma \ref{l:isolatedbottom} is that
  $\cS_0$ is countable.  This
 was proven by White using scaling and monotonicity.

\subsection{Cylindrical approximation}

Given any   $y \in \cS_j \setminus \cS_{j-1}$,   the flow is asymptotic to a shrinking cylinder at $y$, i.e., it looks like a cylinder in $\cC_{j}$ just before $y$.
Moreover, by \cite{CM2},  this limiting cylinder is unique (it has the same axis on each scale).

 The main result of this subsection will show that uniformly $(j,\eta)$-cylindrical subsets of the singular set  satisfy a strong Reifenberg property.  The dimension of the approximating planes will be   the dimension of the affine space for the cylinders.

 \begin{Pro}	\label{t:prA}
 Suppose that
 $S \subset \cS$   satisfies
 \begin{enumerate}
 \item[($\star \eta$)]  For each $\eta > 0$, there exists $\tau_{\eta} > 0$ so that $M_t$ is uniformly $(j,\eta)$-cylindrical on $S$ on the time-scale $\tau_{\eta} $.
  \end{enumerate}
Then:
\begin{enumerate}
\item $S$ has the strong parabolic $j$-dimensional  vanishing Reifenberg property.
\item The associated distribution of $j$-planes is 
$f$-regular for some function $f$.
\end{enumerate}
\end{Pro}

\begin{proof}
We will show that $S$ has the strong parabolic Reifenberg property, where the constant depends on $\eta$ and goes to zero as $\eta$ does.  The vanishing claimed in (1)   as well as the $f$-regularity in (2) will then follow from Corollary \ref{c:techmon} which implies that $\eta$ goes to zero uniformly as we shrink the scale.
 
Given a point $y \in S$, let $\cC_y$ be the cylinder blowup at $y$ (which is unique by \cite{CM2}) and let $V_y$ be the $j$-plane through $y$ that is the axis of  $\cC_y$.   To get the Reifenberg property (1), we will show that   for  $r <  \sqrt{\tau}/2$ 
\begin{align}	\label{e:Ssubr}
	  \cPB_r(y) \cap S \subset \cPT_{\delta \, r}(V_y) \, ,
\end{align}
 where $\delta$ depends on $\eta$ and goes to zero as $\eta$ does.  We will divide \eqr{e:Ssubr} into two parts, where we first show it for the projection onto time and then for the projection onto space.

If $z$  is   in $\left( \cPB_r(y) \cap S \right) \setminus \{ y \}$, let $\cC_z$ and $V_z$ be the corresponding cylinder and $j$-plane through $z$, respectively. 
Without loss of generality, suppose that 
\begin{align}
	t(y) \leq t(z) \, .
\end{align}
 As long as $\eta > 0$ is small enough, it follows that the two spatial regions where the flow is cylindrical (one centered at $y$ and one at $z$)  overlap when $t= t(z) - 4r^2$.  
 Thus, the flow is close to two cylinders on the overlap, with the radius of each cylinder given in terms of the time to the singularities at $y$ and $z$, respectively.
 The   cylindrical structure about $z$ implies that the flow is a graph over a cylinder of radius  
 \begin{align}
 	\sqrt{2(n-j)}\, (2r) \, .
\end{align}
 On the other hand, the cylindrical structure about $y$ implies that the flow is a graph over a cylinder of radius 
\begin{align}
	\sqrt{2(n-j)}\, \sqrt{ 4r^2 + t(y) - t(z)} \, .
\end{align}
Comparing the two radii (and noting that $t(y) \leq t(z)$), we see that
\begin{align}	\label{e:etaty}
	  ( \sqrt{2(n-j)} + \eta) \, \sqrt{ 4r^2 + t(y) - t(z)}  \geq  (\sqrt{2(n-j)} - \eta)  \, (2r) \, .
\end{align}
In the limit as $\eta \to 0$, \eqr{e:etaty} would imply that $t(y)= t(z)$.
 Given $\gamma > 0$, then we can take $\eta$   small enough so  that \eqr{e:etaty} implies that
\begin{align}	\label{e:loconeHz}
 	 \gamma\, r^2 \geq    |t(y) - t(z)|  \, .
\end{align}
  This   shows that \eqr{e:Ssubr} holds for the projection onto time.

We now look at the flow at time $\bar{t} \equiv t(y) - \frac{r^2}{4}$.  It is convenient to set $\rho^2 =t(z) -  \bar{t}$.  Note that \eqr{e:loconeHz} guarantees that this makes sense for $\gamma$
small
 enough and, in fact, that
\begin{align}	\label{e:compat}
	\left| \rho^2 - \frac{r^2}{4} \right| = \left| t(z) - t(y) \right|  \leq\gamma \,  r^2  \, .
\end{align}
We will choose $\gamma$ small enough so that this implies that $\rho \in (r/4 , 3r/4)$.

Let $\Pi : \RR^{n+1} \times \RR \to \RR^{n+1}$ be   projection from space-time to space.
The $(j,\eta)$-cylindrical property at $y$ (with $s= \frac{r^2}{4}$) gives that
\begin{itemize}
	\item $\Psi_{y,\frac{2}{r}} \, \left( B_{\eta^{-1} \, \frac{r}{2} }(\Pi(y)) \cap M_{\bar{t}} \right)$ is a graph over $\cC_y$ of a function with $C^1$ norm at most $\eta$.
\end{itemize}
On the other hand, the $(j,\eta)$-cylindrical property at $z$ (with $s= \rho^2$) gives
\begin{itemize}
	\item $\Psi_{z,\frac{1}{\rho}} \, \left( B_{\eta^{-1} \rho}(\Pi(z)) \cap M_{\bar{t}} \right)$ is a graph over $\cC_z$ of a function with $C^1$ norm at most $\eta$.
\end{itemize}
We will choose  $\eta \in (0,1/4)$ so that $\eta^{-1} - 2 > \eta^{-1}/2$.  Thus,  since $\Pi(z) \in B_r (\Pi(y))$, we have
\begin{align}
	B_{\eta^{-1} \frac{r}{4}} (\Pi(z)) &\subset B_{(\eta^{-1} - 2) \frac{r}{2}} (\Pi(z)) \subset B_{\eta^{-1} \, \frac{r}{2} }(\Pi(y)) \, , \\
	B_{\eta^{-1} \frac{r}{4}} (\Pi(z)) & \subset B_{\eta^{-1} \, \rho} (\Pi(z)) \, .
\end{align}
In particular, we know that $B_{\eta^{-1} \frac{r}{4}} (\Pi(z)) \cap M_{\bar{t}}$ is a graph over both  $\Psi_{z,\rho}(\cC_z)$ and $\Psi_{y,\frac{r}{2}} (\cC_y)$.  
(We have now dilated the cylinders instead of $M_{\bar{t}}$.) As a consequence, 
we have
\begin{align}
	B_{\eta^{-1} \frac{r}{4}} (\Pi(z)) \cap \Psi_{z,\rho} (\cC_z) \subset  T_{\eta (\rho+\frac{r}{2})} (\Psi_{y,\frac{r}{2}} (\cC_y))  
	\, .
\end{align}
As long as $\eta > 0$ is small enough (depending on $\delta$), it follows that $\Pi(z)$ lies in the $(\delta r)$-tubular neighborhood of $V_y$.
 This completes the proof of the strong parabolic Reifenberg property.
This also shows  that the $j$-planes $V_y$ and $V_z$ must be close and, in fact, the distance between them goes to zero uniformly as the distance from $y$ to $z$ goes to zero; this shows the $f$-regularity.
\end{proof}

\subsection{The  strata are cylindrical}

The following proposition shows that the top strata $\cS_{n-1} \setminus \cS_{n-2}$ is always  $(n-1,\eta)$-cylindrical  on some  time-scale $\tau > 0$, with a similar statement for the lower strata.

\begin{Pro}	\label{p:uniapprox}
We have:
\begin{itemize}
\item
Given $\eta > 0$, there exists $\tau > 0$ (depending also on the flow $M_t$) so that
$M_t$ is 
$(n-1,\eta)$-cylindrical   on $\cS_{n-1} \setminus \cS_{n-2}$ on the time-scale $\tau$.
\item
Given $\eta > 0$, $j$, and $\epsilon > 0$, there exists $\tau > 0$ (depending also on the flow $M_t$) so that
$M_t$ is 
$(j,\eta)$-cylindrical   on $\left( \cS_j \setminus \cS_{j-1} \right) \setminus \cPT_{\epsilon} \left( \cS \setminus \cS_j \right)$ on the time-scale $\tau$.
\end{itemize}

\end{Pro}

\begin{proof} 
Let $\epsilon > 0$ (depending on $\eta$) be given by Corollary \ref{c:techmon}.

Given any point $(x_0, t_0) \in \cS_{n-1} \setminus \cS_{n-2}$, there must exist some $\tau_0 > 0$ so that 
\begin{align}
	F_{x_0 , t_0 + 2\tau_0} (M_{t_0+2\tau_0}) \leq \Theta_{x_0 , t_0} + \frac{\epsilon}{2} \, .
\end{align}
In particular,   Corollary \ref{c:techmon} gives $r_0 > 0$ so that each point in 
\begin{align}
	\cPB_{r_0} (x_0 , t_0) \cap \cS_{n-1} \setminus \cS_{n-2}
\end{align}
 is $(n-1,\eta)$-cylindrical at time-scale $\tau_0$.

Since the set $\cS_{n-1} \setminus \cS_{n-2}$ is compact by Lemma \ref{l:unicone}, it follows that it can be covered by a finite collection of balls $\cPB_{r_i} (x_i , t_i)$ where $\cS_{n-1} \setminus \cS_{n-2}$ is $(n-1,\eta)$-cylindrical on time-scale $\tau_i > 0$.  The first claim follows with $\tau = \min_i \tau_i$.

The second claim follows similarly.
\end{proof}

\vskip2mm
We conclude that $\cS_{n-1} \setminus \cS_{n-2}$ has a strong Reifenberg property, with a similar statement for the lower strata:

\begin{Cor}	\label{c:pr}
(1) and (2) in Proposition \ref{t:prA} hold for both
\begin{itemize}
\item $S = \cS_{n-1} \setminus \cS_{n-2}$ with $j=(n-1)$.  
\item   $S = \left( \cS_k \setminus \cS_{k-1} \right) \setminus \cPT_{\epsilon} \left( \cS \setminus \cS_k \right)$
for  each $\epsilon > 0$ with  $j=k$.
\end{itemize}
\end{Cor}

\begin{proof}
This follows by combining Proposition \ref{p:uniapprox} (which gives that the sets are cylindrical) and  
Proposition \ref{t:prA} (which gives that (1) and (2) hold for cylindrical sets).
\end{proof}

\subsection{The structure of the singular set}

 The next theorem records the   properties of the singular set $\cS$ in detail.  
 
 \begin{Thm}	\label{t:detail}
 Let $M_t \subset \RR^{n+1}$ be a MCF  with only cylindrical singularities starting at a closed smooth embedded hypersurface.  The top strata $\cS_{n-1} \setminus \cS_{n-2}$
 satisfies:
 \begin{itemize}
 \item It is contained in finitely many $(n-1)$-dimensional Lipschitz submanifolds  and, thus, has finite $\cPH_{n-1}$ measure.
 \item It  has the strong parabolic $(n-1)$-dimensional vanishing Reifenberg property.
 \item  It is locally the graph of a $2$-H\"older function on space.
 \end{itemize}
Moreover, $ \cS_{n-2} $ has dimension at most $n-2$ and,
 for each $k \leq  n-2$, the set $\cS_k \setminus   \cS_{k-1}$ can be written as the countable union  $\cup_{i=1}^{\infty} \, S_{k,i}$ where each $S_{k,i}$ satisfies:
 \begin{itemize}
\item $S_{k,i}$ is contained in finitely many $k$-dimensional Lipschitz submanifolds.
\item $S_{k,i}$ has the strong parabolic $k$-dimensional vanishing Reifenberg property.
\item $S_{k,i}$ is locally the graph of a $2$-H\"older function on space.
\end{itemize}

 \end{Thm}
 
 \begin{proof}
 By  Corollary \ref{c:pr}, the properties (1) and (2)   in Proposition \ref{t:prA} hold for  $ \cS_{n-1} \setminus \cS_{n-2}$ and $j=(n-1)$.  This gives the second claim for $ \cS_{n-1} \setminus \cS_{n-2}$.  The first and third  claims then follow  from Theorem \ref{t:Lipgraph}. 

The properties of the lower strata follow similarly by applying the second claim in Corollary \ref{c:pr} with $\epsilon = 2^{-i}$ and letting $i \to \infty$.
  \end{proof}
  
  \vskip2mm
 Theorem \ref{t:detail} proves a strong form of rectifiability of the top strata and countable rectifiability for each of the lower strata (at no point does one need to disregard a set of measure zero as is usually done in the definition of rectifiability).

 \begin{proof}[Proof of Theorem \ref{t:main}]
 This follows from Theorem \ref{t:detail}.
 \end{proof}

 \subsection{Proof of Theorem \ref{t:main2}}

   The $k$-dimensional {\emph{parabolic Hausdorff measure}} $\cPH_k$ 
of a set $S \subset \RR^{n+1} \times \RR$
is the $k$-dimensional Hausdorff measure with respect to the parabolic metric.  When $S$ is contained in a time-slice, this agrees
with the usual  $k$-dimensional Hausdorff measure.  In contrast, the
time axis has parabolic Hausdorff dimension two.

\vskip2mm
The next elementary lemma relates the parabolic Hausdorff measure (as a subset of space-time) of a graph of a $2$-H\"older function  to the Euclidean Hausdorff measure of its projection to space.

\begin{Lem}	\label{l:sholder}
Suppose that $\Omega \subset \RR^{n+1}$ and $u: \Omega \to \RR$ is $2$-H\"older continuous.  Then for every $k$, there exists a constant $C$ depending on $k$ and the H\"older constant so that 
 \begin{align}
 	\cH_k(\Omega)\leq \cPH_{k} ( {\text{Graph}}_u) \leq C\, \cH_k (\Omega) \, .
 \end{align}
\end{Lem}

\begin{proof}
Since $u$ is $2$-H\"older, the   map $x \to (x,u(x))$ is Lipschitz with respect to the Euclidean distance on the domain and parabolic distance in the target.  The claim now follows.
\end{proof}

We will say that  a function $u$ on $\Omega \subset \RR^{n+1}$ is {\emph{$2$-H\"older with vanishing constant}} if there is a continuous increasing function $\gamma: [0,\infty) \to [0, \infty)$ with $\gamma (0) = 0$ so that
 \begin{align}	
 	|u(x) - u(y)| \leq \gamma (\epsilon) \, |x-y|^2  {\text{ if }} x,y \in \Omega {\text{ and }} |x-y| \leq \epsilon \, .
\end{align}
In particular, the graph functions in Theorem \ref{t:detail} are automatically  $2$-H\"older with vanishing constant because the graphs
 satisfy the vanishing Reifenberg property.

The next lemma gives a condition which ensures that a  graph is contained in a time-slice, in contrast to the example in Figure \ref{f:2graph}.

\begin{Lem}  \label{l:timeslice}
If $S \subset \RR^{n+1}\times \RR$ is   the graph of a $2$-H\"older   function $u$ with vanishing constant on a subset $\Omega \subset \RR^{n+1}$  with   $\cH_2 (\Omega) < \infty$, then:
\begin{itemize}
\item $\cH_1(t(S))=0$,
\item   if $S$ is connected, then it is contained in a time-slice.
\end{itemize}
\end{Lem}

\begin{proof}
Given $\epsilon>0$, we can cover $\Omega$ by balls $B_{r_i}(x_i)$ with $0<r_i\leq \epsilon$, $x_i \in \Omega$, and 
\begin{align}
\sum_i r_i^2\leq \cH_2(\Omega)+\epsilon\, .
\end{align}
Moreover, $|u(y)-u(x_i)|\leq \gamma (\epsilon)\,r_i^2$ for all $y\in B_{r_i}(x_i) \cap \Omega$ and, hence,
\begin{align}
\cH_1(t(S))\leq \sum_i\gamma (\epsilon)\,r_i^2\leq \gamma (\epsilon)\,\left(\cH_2(\Omega)+\epsilon\right)\, .
\end{align}
Letting $\epsilon\to 0$ gives the first claim.   The second claim follows from the first since $t(S)$ is connected if $S$ is.  
\end{proof}

The same argument gives that if $S\subset \RR^{n+1}\times \RR$ is a $2$-H\"older graph with vanishing constant
 and $\cPH_1(S)<\infty$, then $\cPH_1(t(S)) = \cH_{ \frac{1}{2} } (t(S))=0$.

  \begin{proof}[Proof of Theorem \ref{t:main2}]
(A) follows from Theorem \ref{t:detail}.
  
 To see (B), let $S_{k,i}$ be as in Theorem \ref{t:detail}.  Each intersection $S\cap S_{k,i}$ has  finite $\cPH_2$ measure and is a $2$-H\"older graph with vanishing constant.  Therefore, by Lemma \ref{l:timeslice},  
 \begin{align}
 	\cH_1(t(S\cap S_{k,i}))=0 {\text{ for each $k$ and $i$}}.  
\end{align}
Similarly, we have that
 $\cH_1 ( t(S \setminus \cS_{n-2})) = 0$.
  Since $S\subset (\cS \setminus \cS_{n-2})\cup_{k,i}S_{k,i}$ where the union is taken over countably many sets, it follows that $\cH_1(t(S)) =0$.  Furthermore, if   $S$ is also connected,  then
  so is $t(S)$ and  $S$ must be contained in a time-slice.  
 \end{proof}

  \begin{proof}[Proof of Corollary \ref{c:main2}]
  By Theorem \ref{t:detail}, the singular set $\cS$ has finite $\cPH_2$ measure when $n=2$ or $n=3$.  The corollary now follows from the 
second part  of Theorem \ref{t:main2}.
 \end{proof}

 \section{Local cone property}		\label{s:loccone}

The results of the rest of the paper are not used in the proofs of any of the results stated in the introduction.  

The proof of the rectifiability theorem (Theorem \ref{t:main}) very strongly used the 
uniqueness of tangent flows.
In this section, we will give  weaker criteria that are sufficient for Theorem \ref{t:main2} and do not require the full strength of uniqueness.  Moreover, these criteria are well-suited for other parabolic problems where uniqueness is not known.

We begin by introducing a  scaling condition for subsets of space-time that is natural in parabolic problems such as the heat equation or MCF.  
This    automatically holds for the singular set of a MCF with cylindrical singularities,   
 but is more general.  Moreover,   it  immediately implies that nearby singularities happen at essentially  the same time, as in the examples of shrinking cylinders and tori of revolution.    This condition has two equivalent forms:  The forward and  backward parabolic cone properties.     
   
\subsection{The parabolic cone property}

One way of characterizing 
a Euclidean cone is  that it is invariant under  scaling  $x \to \lambda x$.  In parabolic problems, like the heat equation or MCF, the natural scalings of space-time 
$\RR^{n+1} \times \RR$ are  parabolic dilations about the origin
\begin{align}	\label{e:paracone}
	(x,t) \to (\lambda x , \lambda^2 t) \, .
\end{align}
A parabolic cone  is a subset of space-time  that is invariant under parabolic dilations   (or under parabolic dilations about another point).  For example, the set
$\{ |x|^2 = |t| \}$ is a parabolic double cone; it is   a double paraboloid with the two paraboloids tangent to each other.

To define the parabolic cones that we will use here, 
given a point $y \in \RR^{n+1} \times \RR$, let $\Pi (y)$ be the projection to $\RR^{n+1}$ and $t(y)$   the projection onto the time axis.
Let $\bC_{ \gamma}(z) \subset \RR^{n+1}\times \RR$ be the  parabolic   cone
 centered at $z \in \RR^{n+1}\times \RR$ defined by (see Figure \ref{f:cone})
\begin{align}
	\bC_{\gamma }(z) = \{ y \in \RR^{n+1} \times \RR \, | \, \gamma\,\left| \Pi (y) - \Pi (z) \right|^2 \geq \left| t(y) - t(z) \right|
	    	\} \, .
\end{align}
Thus  $\bC_{\gamma }(z) $ is the region between   two tangent paraboloids and the constant $\gamma$ measures the ``angle of the parabolic cone''.
As $\gamma$ goes to $0$, the region collapses to a time-slice.

  \begin{figure}[htbp]		
\centering\includegraphics[totalheight=.4\textheight, width=.8\textwidth]{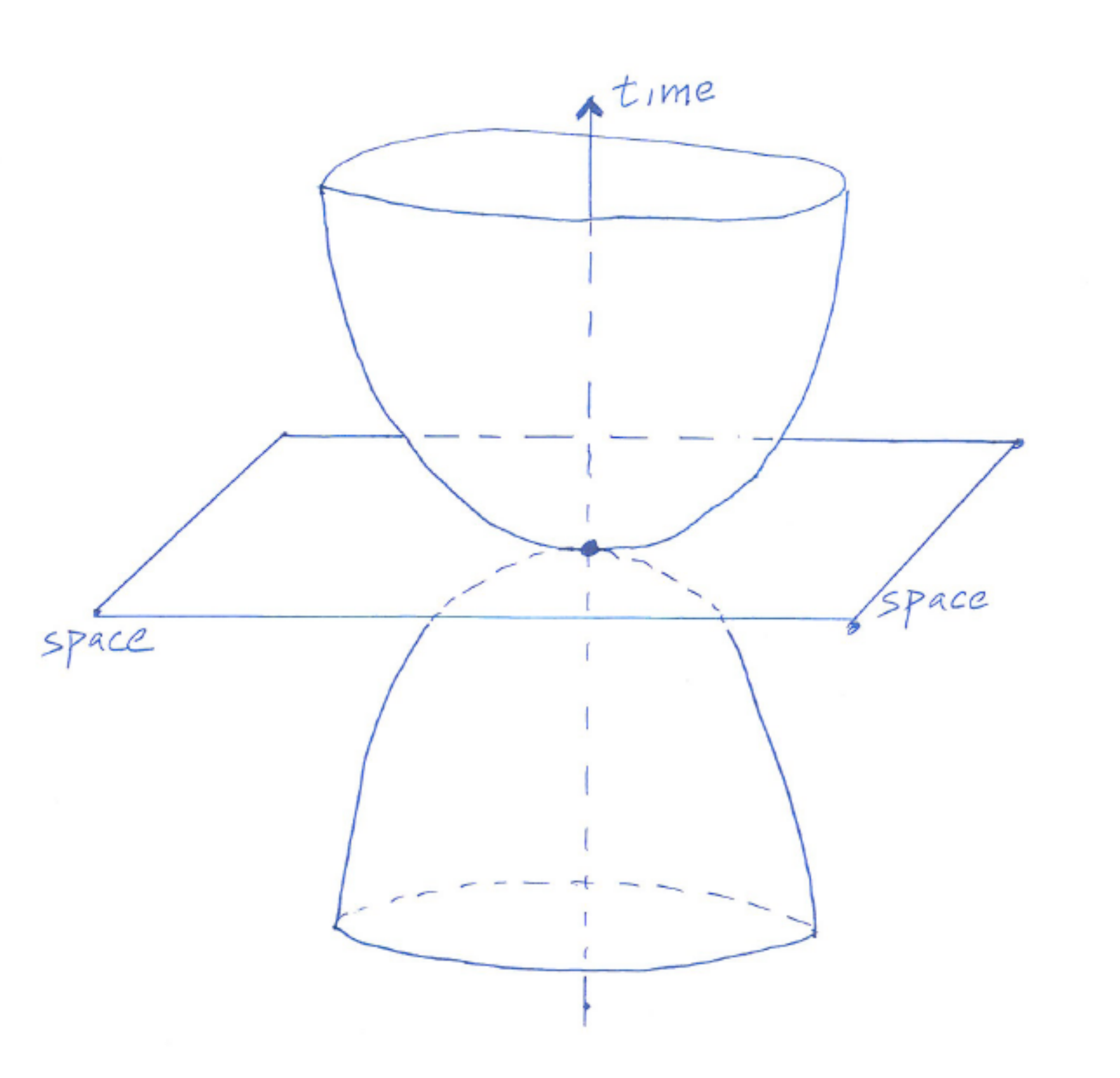}
\caption{The parabolic   cone $C_{\gamma}(0)$.}
 \label{f:cone}
  \end{figure}

 A set  satisfies
the parabolic cone property at a point if it sits between these two tangent paraboloids that make up the parabolic cone.   We say that
  a set $S$ has the $\gamma$-{\it local parabolic cone property}\footnote{Cf. lemma ${\text{I}}.1.2$ in \cite{CM3} and section ${\text{III}}.2$ in
\cite{CM4}.} if  there exists $r_0 > 0$ so that
\begin{align}	\label{e:localcone}
	  \cPB_{r_0}(z)  \cap S   \subset \bC_{\gamma}(z) {\text{ for all }} z \in S  \, ;
\end{align}
see Figure \ref{f:cone2}.
We say that $S$ has the {\emph{vanishing local parabolic cone property}} if $\gamma=\gamma (r_0)\to 0$ as $r_0\to 0$.

  \begin{figure}[htbp]		
\centering\includegraphics[totalheight=.4\textheight, width=.8\textwidth]{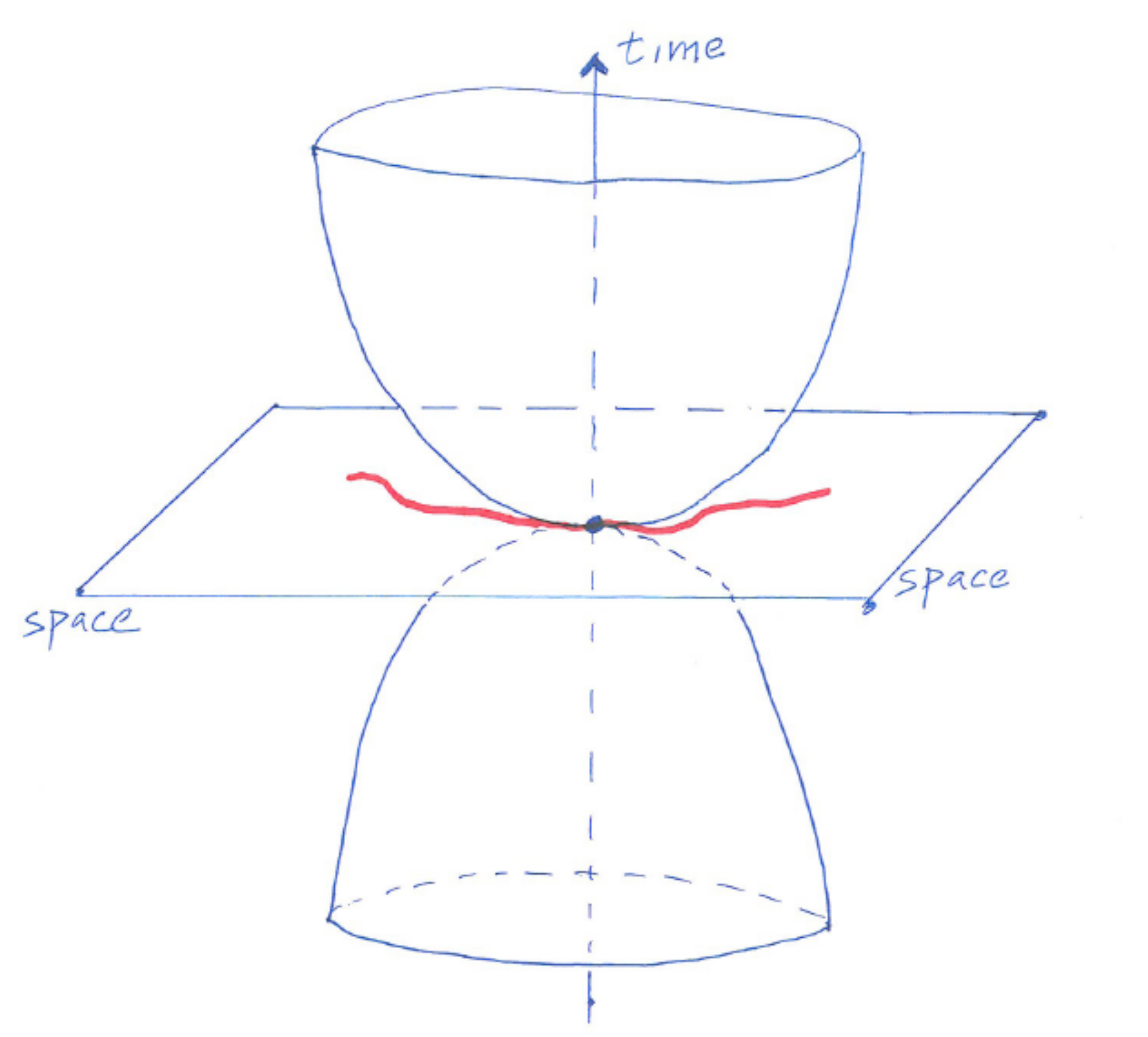}
\caption{The red set lies in the parabolic   cone $C_{\gamma}(0)$.}
 \label{f:cone2}
  \end{figure}

We observe next that if a set satisfies a half-cone property, then it automatically satisfies the full-cone property (we state this for the forward half-cone; 
the same is true for the backward half-cone); see Figure \ref{f:conehalf}.

\begin{Lem}	\label{l:halfcone}
If  there exists $r_0 > 0$ so that
\begin{align}	\label{e:localconeHALF}
	  \cP B_{r_0}(z)  \cap S \cap \{ y\, | \, t(y) > t(z) \}  \subset \bC_{\gamma}(z)  {\text{ for all }} z \in S\, ,
\end{align}
then $S$ has the $\gamma$-{\it local parabolic cone property}.
\end{Lem}

\begin{proof}
Suppose that  $y$ and $z$ are  points in $S$ with $t(y) \ne t(z)$.  We must show that
\begin{align}	\label{e:loconeH}
 	\gamma\,\left| \Pi (y) - \Pi (z) \right|^2\geq |t(y) - t(z)| \, .
\end{align}
If $t(y) > t(z)$, then this follows from \eqr{e:localconeHALF} at $z$.  If $t(y) < t(z)$, then \eqr{e:loconeH} follows from
 \eqr{e:localconeHALF} at $y$.

\end{proof}

  \begin{figure}[htbp]		
\centering\includegraphics[totalheight=.3\textheight, width=.8\textwidth]{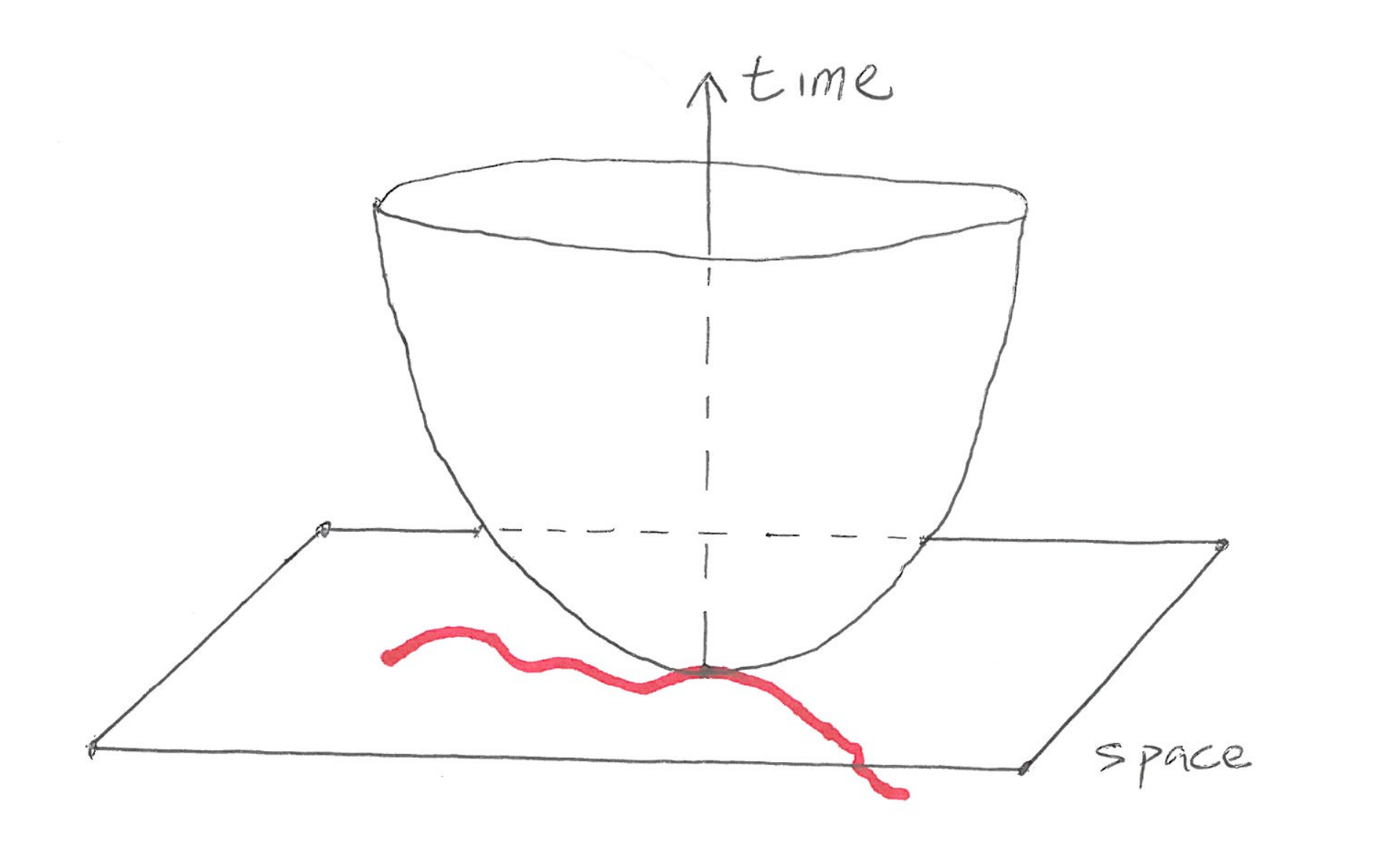}
\caption{The red set lies in the forward parabolic cone.}
 \label{f:conehalf}
  \end{figure}

 The next proposition shows that a set  satisfying the parabolic cone property is a $2$-H\"older graph $(x,u(x))$ where $x$ is in space and $t=u(x)$.

\begin{Pro}	\label{p:alphagraph}
If $S \subset \RR^{n+1}\times \RR$ has the $\gamma$-local parabolic cone property, then $S$ 
is locally the graph{\footnote{The function $u$ may be multi-valued, but the projection from $S$ to $\{ t=0 \}$ is a finite-to-one covering map.}} of a $2$-H\"older regular function $u$ with H\"older constant $\gamma$
\begin{align}
	 S=\{(x,u(x))\,|\, x \in \Omega  \subset \{ t= 0 \} \} = {\text{Graph}}_u \, .
\end{align}
\end{Pro}

\begin{proof}
Fix a parabolic ball where the parabolic cone property holds; we will show that $S$ is a graph in this ball.  From now on, we will work only inside this ball.

Given $y$ and $z$ in $S$ (inside this ball),  
the local parabolic cone property gives
\begin{align}	\label{e:locone}
 	\gamma\,\left| \Pi (y) - \Pi (z) \right|^2\geq |t(y) - t(z)| \, .
\end{align}
It follows immediately that the projection $\Pi : S \to \{ t = 0 \}$ is one to one and, thus, that $S$ is a graph of a function $u$ defined by 
\begin{align}	\label{e:defineu}
	u (\Pi (z)) = t(z) \, .
\end{align}
over a subset $\Omega = \Pi (S) \subset \{ t = 0 \}$.  The $2$-H\"older bound follows from \eqr{e:locone}.
\end{proof}

The next corollary gives a condition which ensures that a  set is contained in a time-slice.
 
\begin{Cor}  \label{c:timeslice}
If $S \subset \RR^{n+1}\times \RR$ has the vanishing parabolic cone property and $\cPH_2(S)<\infty$, then:
\begin{itemize}
\item $\cH_1(t(S))=0$,
\item   if $S$ is connected, then it is contained in a time-slice.
\end{itemize}
\end{Cor}

\begin{proof}
By Proposition \ref{p:alphagraph} and Lemma \ref{l:sholder}, $S$ is locally the graph of a $2$-H\"older regular function $u$ on some domain $\Omega$ in space with finite $\cH_2$ measure.  
The vanishing parabolic cone property implies that $u$ is $2$-H\"older with vanishing constant.
The corollary now follows from Lemma \ref{l:timeslice}.
\end{proof}

The same argument gives that if $S\subset \RR^{n+1}\times \RR$ has the vanishing parabolic cone property and $\cPH_1(S)<\infty$, then $\cPH_1(t(S)) = \cH_{ \frac{1}{2} } (t(S))=0$.

 \section{Rapid clearing out}

In this section, 
we show that the entire flow is rapidly 
clearing out {\emph{after}} a cylindrical singularity; this will not be used elsewhere.

\begin{Thm}		\label{t:halfcone}
There exist
 constants $T , \omega > 1$ so that if  $M_t$ is $(j,\eta)$-cylindrical at $(x_0 , t_0)$ on the time-scale $\tau > 0$ for some $\eta < 1$, then for $s \in (0,\tau)$
\begin{align}	\label{e:localconeHALFm}
	  B_{\eta^{-1} \, \frac{\sqrt{s}}{2} }(x_0)  \cap M_t  = \emptyset {\text{ for }} t \in \left( t_0 + (T-1)s ,  \, t_0 + \left( \frac{\eta^{-2} - 4\omega^2}{4 \omega^2} \right) \, s \right) \, .
\end{align}
\end{Thm}

\vskip2mm
Note that the proposition only has content when $\eta$ is small enough that $ T  \leq \frac{\eta^{-2} }{4 \omega^2} $.

\vskip2mm
The key for the theorem is a local estimate for the Gaussian areas.  Since the cylinder has sub-Euclidean volume growth, its Gaussian surface area $F_{x,t}$ is small when $t$ is large.  The next lemma observes that this is true for any hypersurface $\Sigma$ close to a cylinder.

\begin{Lem}	\label{l:cyl}
There exist $T> 1$ and $\omega > 1$ so that if $\lambda (\Sigma) \leq \lambda_0$ and 
\begin{itemize}
\item  $B_{\eta^{-1}} \cap \Sigma$ is a $C^1$ graph over a cylinder $\cC \in \cC_j$ with norm at most one,
\end{itemize}
then $ F_{x,t} (\Sigma) \leq \frac{1}{2}$ as long as  $t \geq T$ and   $|x| + \omega \sqrt{t} \leq \eta^{-1}$.
\end{Lem}

\begin{proof}
This follows from the sub-Euclidean volume growth of the cylinders.  Namely, there exists a constant $c_n$ depending only on $n$ so that 
\begin{align}	\label{e:lowg}
	\sup \, \left\{ \Vol ( B_R (x) \cap \Sigma ) \, | \, |x| + R \leq \eta^{-1} {\text{ and }} R \geq 1 \right\} \leq c_n \, R^{n-1} \, .
\end{align}
Suppose  that $t \geq T > 1$,  $x \in \RR^{n+1}$ and  $\omega > 1$ satisfy $|x| + \omega \sqrt{t} \leq \eta^{-1}$.  We have
\begin{align}
	F_{x,t} (\Sigma) &\leq 
	\left( 4 \pi t\right)^{ - \frac{n}{2} } \,  \Vol (B_{ \omega \sqrt{t} }(x) \cap \Sigma )
	+ \left( 4 \pi t\right)^{ - \frac{n}{2} } \,
	\int_{\Sigma \setminus B_{ \omega \sqrt{t} }(x)} \e^{ - \frac{ |y-x|^2}{4t} } \, dy \, .
\end{align}
Since $|x| + \omega \sqrt{t} \leq \eta^{-1}$, we can use \eqr{e:lowg} to 
  estimate the first term by
\begin{align}	\label{e:firstterm}
	\left( 4 \pi t\right)^{ - \frac{n}{2} } \,  \Vol (B_{ \omega \sqrt{t} }(x) \cap \Sigma ) &\leq \left( 4 \pi t\right)^{ - \frac{n}{2} } \, c_n \, \omega^{n-1} \, t^{ \frac{n-1}{2}}
	 \leq c_n \, \omega^{n-1} \, \left( 4 \pi \right)^{ - \frac{n}{2} } \, T^{ - \frac{1}{2} } \, .
\end{align}
For the second term, we have
\begin{align}
	&\int_{\Sigma \setminus B_{ \omega \sqrt{t} }(x)} \e^{ - \frac{ |y-x|^2}{4t} } \, dy = \sum_{k=1}^{\infty} 
	\int_{B_{(k+1) \,  \omega \sqrt{t} }(x) \cap \Sigma \setminus B_{k\, \omega \sqrt{t} }(x)} \e^{ - \frac{ |y-x|^2}{4t} } \, dy   \\
	&\qquad \leq \sum_{k=1}^{\infty} 
	 \Vol (B_{(k+1) \,  \omega \sqrt{t} }(x) \cap \Sigma) \,  \e^{ - \frac{ \omega^2 \, k^2}{4} }  \leq C \, \lambda_0 \, t^{ \frac{n}{2} } \, \omega^n \, \sum_{k=1}^{\infty} 
	  (k+1)^{n}    \,  \e^{ - \frac{ \omega^2 \, k^2}{4} } 
	 \notag \, ,
\end{align}
where $C$ depends on $n$.
Thus, we can take $\omega$ large enough so that
\begin{align}
	\left( 4 \pi t\right)^{ - \frac{n}{2} } \,
	\int_{\Sigma \setminus B_{ \omega \sqrt{t} }(x)} \e^{ - \frac{ |y-x|^2}{4t} } \, dy \leq \frac{1}{4} \, , 
\end{align}
and this same $\omega$ works independently of $t \geq 1$.  Finally, now that we have chosen $\omega$, we choose $T$ large enough to make 
\eqr{e:firstterm} at most $\frac{1}{4}$.

\end{proof}

The  argument in the proof of Lemma \ref{l:cyl} works more generally for     sub-Euclidean volume growth.  It does not work when the volume growth is Euclidean.

 \begin{proof}[Proof of Theorem \ref{t:halfcone}]
 To simplify notation, translate in space-time so that $x_0=0$ and $t_0 =0$.

 Let $T$ and $\omega$ be given by Lemma \ref{l:cyl}.  For each $s \in (0,\tau)$, we have that
  $B_{\eta^{-1}}  \cap \Psi_{0,\frac{1}{\sqrt{s}}} \, \left(  M_{-s} \right)$ is a graph over a fixed cylinder in $\cC_j$ of a function with $C^1$ norm at most $\eta$.
  In particular, since $\lambda (M_t) \leq \lambda_0$, Lemma \ref{l:cyl} gives that
 \begin{align}
	\sup \, \left\{ F_{x,t} (\Psi_{0,\frac{1}{\sqrt{s}}} \, \left(  M_{-s} \right)) \, | \,  t \geq T {\text{ and }} |x| + \omega \sqrt{t} \leq \eta^{-1} \right\} \leq \frac{1}{2} \, .
\end{align}
Restating this in terms of $M_s$, we get for $s \in (\tau , 0)$ and $|x| \leq \frac{1}{2} \, \eta^{-1} \, \sqrt{s}$  that
\begin{align}		\label{e:Flow}
	 F_{x,t} (  M_{-s}) \leq \frac{1}{2}   {\text{ if }}  Ts \leq  t      \leq
	  \frac{s}{4 \eta^{2}\, \omega^2} 
	  \, .
\end{align}
Combining this with the monotonicity \eqr{e:huisken} gives    
\begin{align}
	 \Theta_{x,t-s} \leq \frac{1}{2} {\text{ for }} Ts \leq  t      \leq
	  \frac{s}{4 \eta^{2}\, \omega^2}   \, .
\end{align}
The proposition follows since $\Theta_{x,t}  \geq 1$
 for $x$   in the support of $M_t$.   
 \end{proof}

\end{document}